\documentclass[12pt,leqno,a4paper,final]{amsart}
\usepackage{amsmath, amsthm, verbatim, amsfonts, amssymb, color}
\usepackage{mathrsfs}
\usepackage{dsfont}
\usepackage{esint}
\usepackage[a4paper,margin=3cm]{geometry}

\usepackage{hyperref}

\theoremstyle{plain}
\newtheorem{theorem}{Theorem}[section]
\newtheorem{corollary}[theorem]{Corollary}
\newtheorem{lemma}[theorem]{Lemma}

\theoremstyle{definition}
\newtheorem{remark}[theorem]{Remark}
\newtheorem{example}[theorem]{Example}
\newtheorem{definition}[theorem]{Definition}

\numberwithin{equation}{section}
\renewcommand\labelenumi{\textup{\alph{enumi})}}
\renewcommand\theenumi\labelenumi

\newcommand\nat{\mathds{N}}
\newcommand\integer{\mathds{Z}}
\newcommand\real{{\mathds{R}}}
\newcommand\rn{{\mathds{R}^n}}
\newcommand\comp{{\mathds{C}}}

\newcommand\I{\mathds{1}}
\newcommand\Pp{\mathds{P}}
\newcommand\Ee{\mathds{E}}

\newcommand{\nnorm}[1]{\|#1\|}
\newcommand{\Fcal}{\mathcal{F}}

\newcommand{\Xcal}{\mathcal{X}}

\newcommand\Bcal{\mathscr{B}}

\usepackage{xcolor}

\usepackage{marginnote}
\usepackage{soul}
\usepackage[notref,notcite]{showkeys}

\begin{document}
\title[Subordination and Function Spaces]{\bfseries Bochner's Subordination and Fractional Caloric Smoothing in Besov and Triebel--Lizorkin Spaces}

\date{To appear in Mathematische Nachrichten 2022}

\author[V.~Knopova]{Victoriya Knopova}
\address[V.~Knopova]{TU Dresden\\ Fakult\"{a}t Mathematik\\ Institut f\"{u}r Mathematische Stochastik\\ 01062 Dresden, Germany}
\email{victoria.knopova@tu-dresden.de}

\author[R.L.~Schilling]{Ren\'e L.\ Schilling}
\address[R.L.~Schilling]{TU Dresden\\ Fakult\"{a}t Mathematik\\ Institut f\"{u}r Mathematische Stochastik\\ 01062 Dresden, Germany}
\email{rene.schilling@tu-dresden.de}

\begin{abstract}
    We use Bochner's subordination technique to obtain caloric smoothing estimates in Besov- and Triebel--Lizorkin spaces. Our new estimates extend known smoothing results for the Gau{\ss}--Weierstra{\ss}, Cauchy--Poisson and higher-order generalized Gau{\ss}--Weierstra{\ss} semigroups. Extensions to other function spaces (homogeneous, hybrid) and more general semigroups are sketched.
\end{abstract}

\maketitle

\section{Introduction}

Let $(W_t^f)_{t\geq 0}$ be the $f$-subordinated Gau{\ss}--Weierstra{\ss} semigroup; by this we mean the family of operators which is defined through the Fourier transform
\begin{gather}\label{e-intro}
    \Fcal(W^f_t u)(\xi)
    = e^{-tf(|\xi|^2)}\,\Fcal u(\xi),\quad u\in S(\rn),
\end{gather}
where the function $f:(0,\infty)\to(0,\infty)$ is a so-called Bernstein function, see Section~\ref{sub}. Typical examples are $f(x)=x$ (which gives the classical Gau{\ss}--Weierstra{\ss} semigroup), $f(x)=\sqrt x$ (which gives the Cauchy--Poisson semigroup) or $f(x) = x^\alpha$, $0<\alpha<1$ (which leads to the stable semigroups). In this note we prove the caloric smoothing of (the extension of) $(W_t^f)_{t\geq 0}$ in Besov and Triebel--Lizorkin spaces, see Section~\ref{fun}. ``Caloric smoothing'' refers to the smoothing effect of the semigroup which can be quantified through inequalities of the following form
\begin{gather}\label{e-intro-2}
    C_{f,d}(t) \nnorm{W_t^{f} u \mid A_{p,q}^{s+d}}
    \leq
    \nnorm{u \mid A_{p,q}^s}
    \quad\text{for all $0<t\leq 1$ and $u\in A_{p,q}^s$},
\end{gather}
where $d\geq 0$ is arbitrary, $C_{f,d}(t)$ is a constant depending only on $f$ and $d$, and $C_{f,d}(t)\to 0$ as $t\to 0$; the symbol $A_{p,q}^s = A_{p,q}^s(\rn)$ stands for a Besov space $B_{p,q}^s(\rn)$ or a Triebel--Lizorkin space $F_{p,q}^s(\rn)$.

Results of this type are known for the Gau{\ss}--Weierstra{\ss} semigroup $W_t$, i.e.\ for $f(x)=x$  (see  Triebel \cite[Theorem~3.35]{T20}) and for the generalized Gau{\ss}--Weierstra{\ss} semigroup $W_t^{(m)}$ where $m\in\nat$; these operators are also given through the relation \eqref{e-intro} if we take $f(x) = x^m$, cf.\ \cite[Remark~3.37]{T20}, but note that for $m>1$ this is not a Bernstein function, and recent results by Baaske \& Schmei{\ss}er \cite[Theorem~3.5]{bas-sch17}.

We will use Bochner's subordination technique to prove \eqref{e-intro-2} for arbitrary Bernstein functions $f(x)$ and arbitrary powers $f(x) = x^\beta$, $\beta>0$. The constant $C_{f,d}(t)$ is comparable with $[f^{-1}(1/t)]^{-d/2}$. As an application we generalize the result of Baaske \& Schmei{\ss}er  \cite[Theorem~3.5]{bas-sch17} on the  existence and uniqueness of the mild and strong solutions of a nonlinear Cauchy problem with arbitrary (fractional) powers of the Laplacian $(-\Delta)^\beta$, $\beta\geq 1$.

\section{Function spaces}\label{fun}

Let us briefly recall some notation. $L_p(\rn)$, resp., $\ell_q(\nat_0)$ denote the spaces of $p$th order integrable functions on $\rn$, resp., $q$th order summable sequences indexed by $\nat_0$; we admit $0<p,q\leq\infty$. Since we are always working in $\rn$, we will usually write $L_p$ instead of $L_p(\rn)$. If $p,q< 1$ these spaces are quasi-Banach spaces, their (quasi-)norms are denoted by $\|\cdot\mid L_p\|$ and $\|\cdot\mid\ell_q\|$, respectively. We write
\begin{gather*}
    \|u(k,x)\mid L_p\mid\ell_q(\nat_0)\|
    \quad\Big[\text{resp.\ \ }
    \|u(k,x)\mid\ell_q(\nat_0) \mid L_p\|\Big]
\end{gather*}
to indicate that we take first the $L_p$-norm and then the $\ell_q(\nat_0)$-norm [resp.\ first the $\ell_q(\nat_0)$-norm and then the $L_p$-norm]. Throughout, we use $j,k,m$ for discrete and $x,y,z$ for continuous variables, so there should be no confusion as to which variable is used for the $\ell_q(\nat_0)$-norm or $L_p$-norm.

We follow Triebel \cite[Definition 1.1 and Remark 1.2 (1.14), (1.15)]{T20} for the definition of the scales of Besov- and Triebel--Lizorkin spaces. Let $\Fcal u$ denote the Fourier transform of a function $u$; the extension to the space of tempered distributions $S'(\rn)$ is again denoted by $\Fcal$. Fix some $\phi_0\in C_0^\infty$ such that $\I_{B(0,1)} \leq \phi_0\leq  \I_{B(0,3/2)}$ and set  $\phi_k(x):= \phi_0(2^{-k}x)- \phi_0(2^{-(k+1)}x)$.
Since $\sum_{k=0}^\infty \phi_k (x)=1$, the sequence $(\phi_k)_{k\geq 0}$ is a \textbf{dyadic resolution of unity}. By
\begin{gather*}
    \phi_k (D) u(x)
    := \Fcal^{-1} (\phi_k\Fcal u)(x)
\end{gather*}
we denote the \textbf{pseudo-differential operator} (Fourier multiplier operator) with symbol $\phi_k$. We will also need the \textbf{dyadic cubes} $Q_{J,M}= 2^{-J} M + 2^{-J} (0,1)^n$, where $J\in \integer$, $M\in \integer^n$ and $(0,1)^n$ is the open unit cube in $\rn$.

\begin{definition}\label{fun-3}
Let $(\phi_k)_{k\geq 0}$  be any dyadic resolution of unity.
\begin{enumerate}
\item\label{fun-3-a}
Let $p\in (0,\infty]$, $q\in (0,\infty]$ and $s\in \real$. The \textbf{Besov space} $B_{p,q}^s$ is the family of all $u\in S'(\rn)$ such that the following (quasi-)norm is finite:
\begin{gather*}
    \|u\mid B_{p,q}^s \|
    := \| 2^{ks} \phi_k(D) u(x)\mid L_p\mid\ell_q(\nat_0)\|.
\end{gather*}
\item\label{fun-3-b}
Let $p\in (0,\infty)$, $q\in (0,\infty]$ and $s\in \real$. The \textbf{Triebel--Lizorkin space} $F_{p,q}^s$ is the family of all $f\in S'(\rn)$ such that the following (quasi-)norm is finite
\begin{gather*}
    \|u\mid F_{p,q}^s \|
    := \| 2^{ks} \phi_k(D) u(x)\mid\ell_q(\nat_0) \mid L_p\|.
\end{gather*}
\item\label{fun-3-c}
Let $p=\infty$, $q\in (0,\infty)$ and $s\in \real$. The \textbf{Triebel--Lizorkin space} $F_{\infty,q}^s$ is the family of all $f\in S'(\rn)$ such that the following (quasi-)norm is finite
\begin{gather*}
    \|u\mid F_{\infty,q}^s \|
    :=\sup_{J\in\nat_0, \, M\in \integer^n} 2^{Jn/q} \left(\int_{Q_{J,M}}  \sum_{k= J}^\infty 2^{ksq} |\phi_k(D) u(x)|^q \, dx \right)^{1/q}.
\end{gather*}
\item\label{fun-3-d}
Let $p=q=\infty$ and $s\in \real$. The \textbf{Triebel--Lizorkin space} $F_{\infty,\infty}^s$ is the family of all $f\in S'(\rn)$ such that the following norm is finite
\begin{gather*}
    \|u\mid F_{\infty,\infty}^s \|
    := \sup_{J\in\nat_0, \, M\in \integer^n} \sup_{x\in Q_{J,M}} \sup_{k\geq J}  2^{ks}|\phi_k(D) u(x)|.
\end{gather*}
\end{enumerate}
\end{definition}
Note that $F_{\infty,\infty}^s=B_{\infty,\infty}^s$ for all $s\in\real$ and that the norms appearing in Definition~\ref{fun-3}.\ref{fun-3-d} and \ref{fun-3}.\ref{fun-3-a} coincide if $p=q=\infty$: $\|u\mid F_{\infty,\infty}^s \|=\|u\mid B_{\infty,\infty}^s \|$. Definition~\ref{fun-3} does not depend on the choice of $(\phi_k)_{k\geq 0}$ since different resolutions of unity lead to equivalent (quasi-)norms. Various properties of these spaces as well as their relation to other classical function spaces can be found in Triebel \cite{T20}, see also \cite{T83} and \cite{T92}.

Consider the heat kernel (Gaussian probability density) related to the Laplace operator on $\rn$
\begin{equation}\label{pt}
    g_t(x):= \frac{1}{(4\pi t)^{n/2}} e^{-|x|^2/(4t)},\quad t>0,\;x\in\rn.
\end{equation}
We can use $g_t(x)$ to define a convolution operator on the space $B_b$ of bounded Borel functions $u:\rn\to\real$
\begin{equation}\label{Wt}
    W_t u(x) := g_t* u(x) = \int_\rn g_t(y-x) u(y)\, dy.
\end{equation}
For positive $u\geq 0$ the above integral always exists in $[0,\infty]$ and extends $W_t$ to all positive Borel functions.
It is not difficult to see that $g_{t+s}=g_t*g_s$, i.e.\ $(W_t)_{t\geq 0}$ is a semigroup. The operators are positivity preserving ($W_tu\geq 0$ if $u\geq 0$) and conservative ($W_t \I  \equiv 1$). If $u\in S$, then $\Fcal(W_t u)(\xi)= e^{-t |\xi|^2 } \Fcal u(\xi)$. We will need the following simple lemma. We provide the short proof for the readers' convenience.

\begin{lemma}\label{lqest}
Let $(W_t)_{t\geq 0}$ be the Gau{\ss}--Weierstra{\ss} semigroup.
\begin{enumerate}
\item\label{lqest-a}
    $W_t:L_p\to L_p$, $p\in [1,\infty]$ is a contraction, i.e.\
    \begin{gather*}
        \|W_t u\mid L_p\| \leq \|u\mid L_p\|.
    \end{gather*}
\item\label{lqest-b}
    Let $\psi_k(\cdot)$ be a sequence of positive measurable functions on $\rn$ such that $(\psi_k(x))_{k\geq 0}\in\ell_q(\nat_0)$ for some $q\in [1,\infty]$ and all $x\in\rn$. Then
    \begin{gather*}
        \| W_t \psi_k(x)\mid\ell_q(\nat_0)\|
        \leq W_t \|\psi_k(\cdot)\mid\ell_q(\nat_0)\|(x).
    \end{gather*}
\end{enumerate}
\end{lemma}
\begin{proof}
Part~\ref{lqest-a} follows immediately from Jensen's inequality, see \cite[Theorem 13.13]{mims}, for the probability measure $g_t(y)\,dy$:
\begin{align*}
    \|W_t u\mid L_p\|^p
    &= \int_\rn \left| \int_\rn  u(x-y) g_t(y)\, dy \right|^p  dx
    \leq \int_\rn \int_\rn  |u(x-y)|^p g_t(y)\, dy \, dx\\
    &=   \|u\mid  L_p\|^p \int_\rn  g_t(z)\, dz
    = \|u\mid  L_p\|^p.
\end{align*}
If $0<p<1$, the inequality is reversed.

\medskip
In order to prove Part~\ref{lqest-b}, we fix $x$ and pick a sequence $(a_k)_{k\in\nat_0}$ from $\ell_{q'}$ where $q' = \frac{q}{1-q}$. We have
\begin{gather*}
    \left|\left\langle W_t\psi_k(\cdot),a_k \right\rangle\right|
    \leq W_t\langle |\psi_k(\cdot)|,|a_k|\rangle
    \leq W_t \|\psi_k(\cdot)\mid\ell_q\|\cdot\|a_k\mid\ell_{q'}\|.
\end{gather*}
In the estimate we use the fact that $W_t$ is linear and positivity preserving, implying that $u\mapsto W_t u$ is monotone. Taking the supremum over all sequences such that $\|a_k\mid\ell_{q'}\|=1$ gives
\begin{gather*}
    \|W_t\psi_k(x)\mid\ell_q\|
    = \sup_{\|a_k\mid\ell_{q'}\| = 1} \left|\left\langle W_t\psi_k(x),a_k \right\rangle\right|
    \leq W_t\left( \|\psi_k(\cdot)\mid\ell_q\|\right)(x).
    \qedhere
\end{gather*}
\end{proof}

Lemma~\ref{lqest} is the key ingredient for our proof that $W_t$ is a contraction in the scales of Besov- and Triebel--Lizorkin spaces.

The next theorem is well-known for indices $1\leq p,q\leq\infty$. Our elementary proof also covers $0<q<1$ in the case of Besov spaces.

\begin{theorem}\label{contr}
  Let $(W_t)_{t\geq 0}$ be the Gau{\ss}--Weierstra{\ss} semigroup and $s\in \real$.
  \begin{enumerate}
  \item\label{contr-a}
  $\|W_t u \mid B_{p,q}^s\| \leq \|u \mid B_{p,q}^s\|$ for all $p\in [1,\infty]$ and $q\in (0,\infty]$.
  \item\label{contr-b}
  $\|W_t u \mid F_{p,q}^s\| \leq c_n\|u \mid F_{p,q}^s\|$ for all $p,q\in [1,\infty]$ with $c_n=2^n$ if $q=\infty$ and $c_n=1$ otherwise.
  \end{enumerate}
\end{theorem}
\begin{proof}
We use Definition~\ref{fun-3} to introduce the respective (quasi-)norms. Note that the operators $W_t$ and $\phi_k(D)$ commute since their symbols (Fourier multipliers) do not depend on $x$.

\medskip
\ref{contr-a}\ \ Fix $p\in [1,\infty]$, $q\in (0,\infty]$ and $s\in\real$ and let $u\in B_{p,q}^s$. Note that $\phi_k(D)u\in L_p$. Since $W_t$ is a contraction in $L_p$---see Lemma~\ref{lqest}.\ref{lqest-a}---, we get
\begin{align*}
    \|W_t u\mid B_{p,q}^s \|
    &=  \| 2^{ks} W_t  \phi_k (D) u(x)\mid L_p\mid \ell_q(\nat_0)\| \\
    &\leq \| 2^{ks}  \phi_k (D) u(x)\mid L_p \mid \ell_q(\nat_0)\|
    = \|u\mid B_{p,q}^s\|.
\end{align*}
The calculation above uses only $p\geq 1$ and does not impose any restriction on $q>0$ and $s\in\real$.

\medskip
\ref{contr-b}\ \ Fix  $p \in [1,\infty)$, $q\in [1,\infty]$,  $s\in\real$,  and let $u\in F_{p,q}^s(\real)$. Note that $\phi_k(D)u$ is measurable. Using the contractivity properties of $W_t$ from Lemma~\ref{lqest}, we get
\begin{align*}
    \|W_t u\mid F_{p,q}^s\|
    &= \|2^{ks} W_t \phi_k (D)u(x) \mid \ell_q(\nat_0) \mid L_p \| \\
    &\leq \|2^{ks} W_t (|\phi_k (D)u|)(x) \mid \ell_q(\nat_0) \mid L_p \| \\
    &\leq \| W_t (\|2^{ks} |\phi_k (D)u| \mid \ell_q(\nat_0)\|)(x) \mid L_p \| \\
    &\leq \|2^{ks} \phi_k (D)u(x) \mid \ell_q(\nat_0) \mid L_p\|
    = \|u\mid F_{p,q}^s \|.
\end{align*}

We will now consider the case $p=\infty$ and $q\in [1,\infty)$. As before, we write $Q_{J,M}$ for the open cube in $\rn$ with side-length $2^{-J}$ and ``lower left corner'' $2^{-J}M\in\integer^n$. Below we use the notation $\fint_Q u(x)\,dx$ to denote $\mathrm{Leb}(Q)^{-1}\int_Q u(x)\,dx$. We can rewrite the norm for $F_{\infty,q}^s$ as
\begin{equation}\label{snorm}
     \|u\mid F_{\infty,q}^s \|
     = \sup_{J\in\nat_0, M\in\integer^n} \bigg( \fint_{Q_{J,M}} \sum_{k\geq J} 2^{ksq} |\phi_k (D) u|^q \,dx \bigg)^{1/q}.
\end{equation}
In order to estimate the norm $\|W_t u\mid F_{\infty,q}^s\|$, we begin with an auxiliary estimate. Fix $J\in\nat_0$ and $M\in\integer^n$. By Jensen's inequality,
\begin{align*}
    \fint_{Q_{J,M}} |W_t w(x)|^q\, dx
    &\leq \fint_{Q_{J,M}} \int_\rn g_t(x-y) |w(y)|^q \, dy \, dx \\
    &= \int_\rn \fint_{Q_{J,M}} g_t(y) |w(x-y)|^q \, dx \, dy \\
    &\leq \sup_{y\in\rn} \fint_{Q_{J,M}}  |w(x-y)|^q\, dx \cdot \int_{\rn}  g_t(y)\, dy \\
    &=  \sup_{y\in\rn} \fint_{y+Q_{J,M}} |w(x)|^q\, dx.
\end{align*}
The shifted cube $Q:= y+Q_{J,M}$ does, in general, not coincide with any of the $Q_{J,N}$, $N\in\integer^n$. Since $Q$ has side-length $2^{-J}$ it intersects at most $2^n$ of the $Q_{J,N}$, $N\in\integer$. Define
\begin{gather*}
    \lambda_{Q,J,N}:= \frac{\int_{Q\cap Q_{J,N}} |w(x)|^q\, dx}{\int_{Q_{J,N}} |w(x)|^q\,dx},
\end{gather*}
and observe that $0\leq \lambda_{Q,J,N}\leq 1$ and $\sum_{N\in\integer^n} \lambda_{Q,J,N} \leq \sum_{N\in\integer^n} \I_{Q\cap Q_{J,N}}  \leq 2^n$ since the sum contains at most $2^n$ non-zero elements. Since $\mathrm{Leb}(Q)=\mathrm{Leb}(Q_{J,N})$, we get
\begin{align*}
    \fint_{Q} |w(x)|^q\, dx
    = \sum_{N\in\integer^n} \frac 1{\mathrm{Leb}(Q)}\int_{Q\cap Q_{J,N}} |w(x)|^q\, dx
    = \sum_{N\in\integer^n} \lambda_{Q,J,N} \fint_{Q_{J,N}} |w(x)|^q\,dx.
\end{align*}
Moreover, observing that $Q=y+Q_{J,M}$ and $\lambda_{N,J,Q}\leq 1$ we have
\begin{align*}
    \fint_{Q_{J,M}} | W_t w(x)|^q\, dx
    \leq \sup_{y\in\rn} \left(\sum_{N\in\integer^n, Q\cap Q_{J,N}\neq \emptyset} \fint_{Q_{J,N}} |w(x)|^q\,dx\right).
\end{align*}
Now repeat the above calculations with $|W_tw|^q = |W_t\phi_k(D)u|^q$ and $|w|^q = |\phi_k(D)u|^q$, multiplied by $2^{ksq}$ and summed over $k\geq J$. Since we have only positive terms, the summation and integration signs can be freely interchanged. Thus,
\begin{align*}
    \fint_{Q_{J,M}} \sum_{k\geq J} 2^{ksq} |W_t \phi_k(D) u(x)|^q\, dx
    &\leq \sup_{y\in\rn}\left(\sum_{N\in\integer^n, Q\cap Q_{J,N}\neq \emptyset} \fint_{Q_{J,N}} \sum_{k\geq J} 2^{ksq} |\phi_k(D) u(x)|^q\,dx\right)\\
    &\leq 2^n \|u \mid F_{\infty,q}^s\|^q
\end{align*}

Finally, for $p=q=\infty$, the estimate is immediate using Lemma~\ref{lqest} and Definition~\ref{fun-3}.\ref{fun-3-d}.
\end{proof}

\begin{remark}
In the proof of Theorem~\ref{contr} and Lemma~\ref{lqest} we only use the following properties of the semigroup $(W_t)_{t\geq 0}$:
\begin{gather*}
    0\leq u\leq 1\implies 0\leq W_t u\leq 1
    \quad\text{and}\quad
    \phi_k(D) W_t = W_t \phi_k(D).
\end{gather*}
This means that Theorem~\ref{contr} holds for \textbf{every} positivity preserving, Markovian semigroup $T_t$ which is given by a convolution: $T_t u = u*\pi_t$ where $(\pi_t)_{t\geq 0}$ is a convolution semigroup of probability measures on $\rn$. These semigroups can be completely characterized using the Fourier transform. One has, see \cite[Section~3.6]{jacob}
\begin{gather*}
    \Fcal (T_t u)(\xi) = e^{-t\psi(\xi)}\Fcal u(\xi), \quad t>0,\; \xi\in\rn
\end{gather*}
where $\psi:\rn\to\comp$ is a \textbf{continuous, negative definite function} (in the sense of Schoenberg) such that $\psi(0)=0$. All such $\psi$ are uniquely characterized by their L\'evy--Khintchine representation
\begin{gather*}
    \psi(\xi)
    = i\ell\cdot\xi + \frac 12 \xi\cdot Q\xi + \int_{y\neq 0}\left(1-e^{iy\cdot\xi} + iy\cdot\xi \I_{(0,1)}(|y|)\right)\nu(dy)
\end{gather*}
such that $\ell\in\rn$, $Q\in\real^{n\times n}$ is positive semidefinite and $\nu$ is a Radon measure on $\rn\setminus\{0\}$ such that $\int_{y\neq 0} \min\{|y|^2, 1\}\,\nu(dy)<\infty$. Typical examples are $\psi(\xi)=|\xi|^2$ (leading to the Gau{\ss}--Weierstra{\ss} semigroup), $\psi(\xi)=|\xi|$ (leading to the Cauchy--Poisson semigroup), $\psi(\xi)=|\xi|^\alpha$, $0<\alpha<2$ (leading to the symmetric stable semigroups), but also $\psi(\xi)=\log(1+|\xi|)$ and many others. These semigroups appear in the study of L\'evy processes, see  e.g.\ \cite{sato13,jacob,barca}.

It is worth noting that $\psi(\xi)$ can grow at most like $|\xi|^2$ as $|\xi|\to\infty$. Although the multipliers $e^{-t|\xi|^{\beta}}$, $\beta>2$, will lead to semigroups, these semigroups are not any longer positivity preserving.
\end{remark}

\section{Bochner's subordination}\label{sub}

In the paper \cite{bochner49} S.\ Bochner started to study initial-value problems of the form
\begin{equation}\label{sub-e04}
    \left\{\begin{aligned}
        \frac{\partial}{\partial t} u(t,x) &= -f(-\Delta_x) u(t,x),\quad t>0,\; x\in\rn,\\
        u(0,x) &= u_0(x),\quad t=0,\; x\in\rn,
    \end{aligned}\right.
\end{equation}
where $\Delta_x$ denotes the Laplace operator on $\rn$ and $f:[0,\infty)\to[0,\infty)$ is a \textbf{Bernstein function} (see Theorem~\ref{sub-03} below). Typical examples are $f(\lambda) = \lambda^\alpha$, $0<\alpha<1$ or $f(\lambda) = \sqrt{\lambda+c}-\sqrt c$. We may study the problem \eqref{sub-e04} in any of the Banach spaces $L_p, 1\leq p<\infty$ or $C_\infty = \{u\in C : \lim_{|x|\to\infty} u(x) = 0\}$; throughout this section we write just $\Xcal$.

From Bochner's representation theorem for positive definite functions we know that there is a family of probability measures $(\mu_t^f)_{t\geq 0}$ on $[0,\infty)$ such that $\phi_t(\lambda)=\exp(-t f(\lambda))$ is their Laplace transform:
\begin{equation}\label{sub-e05}
    \int_0^\infty e^{-\lambda r}\,\mu_t^f(dr) = e^{-t f(\lambda)},\quad t > 0,\; \lambda\geq 0.
\end{equation}
Since $t\mapsto e^{-t f}$ is continuous and satisfies $e^{-(t+s)f}=e^{-tf} e^{-sf}$, it is clear that $(\mu_t^f)_{t\geq 0}$ is a semigroup w.r.t.\ convolution of measures on $[0,\infty)$ which is vaguely (i.e.\ in the weak-$*$ sense) continuous in the parameter $t>0$. Notice that all vaguely continuous convolution semigroups are uniquely determined by their exponent $f$. We may even characterize all such exponents.

\begin{theorem}[Schoenberg]\label{sub-03}
    A function $f:(0,\infty)\to (0,\infty)$ such that $f(0+)=0$ is the characteristic exponent of a vaguely continuous convolution semigroup if, and only if, one of the following equivalent conditions hold
    \begin{enumerate}
    \item\label{sub-03-a}
        $f$ is a \textbf{\upshape Bernstein function}, i.e.\ $f\in C^\infty(0,\infty)$, $f\geq 0$ and $(-1)^{n-1} f^{(n)}\geq 0$, $n\in\nat$;
    \item\label{sub-03-b}
        $e^{-t f}$ is for each $t>0$ a positive definite function;
    \item\label{sub-03-c}
        $f$ has the following L\'evy--Khintchine representation
        \begin{gather*}
            f(\lambda) = b\lambda + \int_0^\infty (1-e^{-\lambda r})\,\nu(dr),\quad\lambda>0,
        \end{gather*}
        with $b\geq 0$ and a measure $\nu$ on $(0,\infty)$ such that $\int_0^\infty \min\{r,1\}\,\nu(dr)<\infty$.
    \end{enumerate}
\end{theorem}
This is a standard result, see e.g.\ \cite[Chapter~3]{SSV} or Jacob~\cite[Sections~3.9.2--3.9.7]{jacob}. Notice that Bernstein functions are automatically strictly increasing.

\bigskip
Bochner showed that the problem \eqref{sub-e04} is solved by the semigroup
\begin{equation}\label{sub-e06}
    W_t^f u_0(x) := \int_0^\infty W_r u_0(x)\,\mu^f_t(dr)
\end{equation}
where $(W_t)_{t\geq 0}$, $W_t = e^{t\Delta}$, is the Gau{\ss}--Weierstra{\ss} semigroup generated by the Laplacian $\Delta$. The integral appearing in \eqref{sub-e06} is understood in a pointwise sense. Moreover, the family $(W_t^f)_{t\geq 0}$ inherits many properties of the semigroup $(W_t)_{t\geq 0}$: it is a semigroup on the same Banach space $\Xcal$ as $(W_t)_{t\geq 0}$, it is again strongly continuous, contractive, positivity preserving and conservative. The infinitesimal generator of $(W_t^f)_{t\geq 0}$ is a function of the Laplacian $-f(-\Delta)$, e.g.\ in the sense of spectral calculus, see \cite[Chapter~13]{SSV}.

\begin{remark}\label{sub-05}
    The formula \eqref{sub-e06} still makes sense for general strongly continuous contraction semigroups $(T_t)_{t\geq 0}$ on abstract Banach spaces $(\Xcal, \nnorm{\cdot})$. The resulting subordinate semigroup $(T_t^f)_{t\geq 0}$ inherits all essential properties of $(T_t)_{t\geq 0}$ such as strong continuity and contractivity and---if applicable---it preserves positivity and is conservative whenever $(T_t)_{t\geq 0}$ is. Using the L\'evy--Khintchine representation of $f$ it is possible to give an explicit formula of the infinitesimal generator of $(T_t^f)_{t\geq 0}$ as a function of the generator of $(T_t)_{t\geq 0}$, see \cite[Theorem~13.6]{SSV}.
\end{remark}

Let us return to the Gau{\ss}--Weierstra{\ss} semigroup. Recall from \eqref{pt} and \eqref{Wt} that
\begin{gather*}
    \smash[b]{\Fcal(W_t u)(\xi) = e^{-t|\xi|^2} \Fcal u(\xi)}
\intertext{and}
    \smash[t]{W_t u(x) = g_t*u(x) = (4\pi t)^{-n/2} \int_\rn e^{-(x-y)^2/(4t)} u(y)\,dy}
\end{gather*}
whenever these expressions make sense, e.g.\ if $u\in S$ (for the first formula) and $u\in L_p$ or $u\geq 0$ and measurable (for the second).

\begin{lemma}\label{sub-07}
    Let $f:(0,\infty)\to (0,\infty)$ be a Bernstein function. The semigroup $(W_t^f)_{t\geq 0}$ subordinate to the heat semigroup $(W_t)_{t\geq 0}$ satisfies
    \begin{gather*}
        \Fcal(W_t^f u)(\xi) = e^{-t f(|\xi|^2)} \Fcal u(\xi),
        \quad t>0,\; u\in S,
    \end{gather*}
    and if $g_t^f(x) := \int_0^\infty g_r(x)\,\mu^f_t(dr) = \int_0^\infty (4\pi r)^{-n/2} e^{- x^2/(4r)}\,\mu^f_t(dr)$ is the generalized heat kernel,
    \begin{gather*}
        W_t^f u(x) = g_t^f*u(x) = \int_\rn\int_0^\infty (4\pi r)^{-n/2} e^{-(x-y)^2/(4t)} u(y)\,\mu^f_t(dr)\,dy,
        \quad t>0, \; u\in L_p.
    \end{gather*}
\end{lemma}
\begin{proof}
    Taking Fourier transforms on both sides of \eqref{sub-e06} with $u_0=u\in S$ gives
    \begin{align*}
        \Fcal(W^f_{t} u)(\xi)
        = \int_0^\infty \Fcal(W_r u)(\xi)\,\mu_t^f(dr)
        = \int_0^\infty e^{-r|\xi|^2}\,\mu_t^f(dr)\,\Fcal u(\xi)
        = e^{-tf(|\xi|^2)}\,\Fcal u(\xi)
    \end{align*}
    where we use Theorem~\ref{sub-03}. The second assertion follows from a similar Fubini-argument.
\end{proof}

\begin{example}\label{sub-09}
    Let $f(\lambda) = f_\alpha(\lambda) = \lambda^\alpha$ for $\lambda\geq 0$ and $0<\alpha <1$. In this case we write $W^{(\alpha)}$ and $g^{(\alpha)}_t$ instead of $W^f_t$ and $g^{f}_t$.

    The L\'evy--Khintchine representation of $f_\alpha$ is
    \begin{gather*}
        \lambda^\alpha = \frac{\alpha}{\Gamma(1-\alpha)}\int_0^\infty (1-e^{-\lambda r}) r^{-\alpha-1}\,dr
    \end{gather*}
    and the Fourier transform of the generalized heat kernel is
    \begin{gather*}
        (2\pi)^{n/2}\Fcal g_t^{(\alpha)}(\xi) = e^{- t |\xi|^{2\alpha}}.
    \end{gather*}
    It is obvious, that $g_t^{(\alpha)}(y)$ is a function, but only for $\alpha = \frac 12$ there seems to be a closed representation with elementary functions
    \begin{gather*}
        g_t^{(1/2)}(x) = \frac {\Gamma\left(\frac{n+1}{2}\right)}{\pi^{(n+1)/2}} \frac{t}{(t^2+|x|^2)^{(n+1)/2}}.
    \end{gather*}
\end{example}

\begin{remark}\label{sub-11}
    It is possible to associate with every vaguely continuous convolution semigroup of measures $(\mu_t^f)_{t\geq 0}$ on $[0,\infty)$ a random process $(S_t^f)_{t\geq 0}$ such that
    \begin{gather*}
        \Pp(S_t^f \in A) = \mu_t^f(A),\quad A\in\Bcal[0,\infty).
    \end{gather*}
    The processes $(S_t^f)_{t\geq 0}$ are called \textbf{subordinators}. One can show that a subordinator is a random process with stationary and independent increments and right-continuous trajectories $t\mapsto S_t$ (L\'evy process) such that $S_0=0$ and $t\mapsto S_t$ is increasing. This allows us to write for any bounded or positive Borel function $g$
    \begin{gather*}
        \int_0^\infty g(r)\,\mu_t^f(dr)
        \quad\text{as an expected value}\quad
        \Ee \left[g(S_t^f)\right];
    \end{gather*}
    this will be useful later on, in order to calculate certain constants.

    If $f(\lambda)=\lambda^\alpha$, the corresponding process $(S_t^{(\alpha)})_{t\geq 0}$ is usually called an \textbf{$\alpha$-stable subordinator}.
\end{remark}

\section{Fractional caloric smoothing}\label{est}

Let $s\in\real$ and $0<p,q\leq\infty$. Denote by $A_{p,q}^s$ one of the spaces $B_{p,q}^s$ or $F_{p,q}^s$ and write $\nnorm{u\mid A_{p,q}^s}$ for its (quasi-)norm. As before, $(W_t)_{t\geq 0}$ is the heat semigroup. We have seen in Theorem~\ref{contr} that $W_t$ is a contraction in the $B$-scale if $s\in\real$, $1\leq p\leq\infty$, $0<q\leq\infty$ and in the $F$-scale if $s\in\real$, $1\leq p, q\leq\infty$. The following caloric smoothing estimate can be found in \cite[Theorem~3.35]{T20}: {\itshape For every $d\geq 0$ there is a constant $c>0$ such that
\begin{equation}\label{est-e10}
    \nnorm{W_t u \mid A_{p,q}^{s+d}}
    \leq c t^{-d/2} \nnorm{u \mid A_{p,q}^s}
    \quad\text{for all $0<t\leq 1$ and $u\in A_{p,q}^s$}.
\end{equation}}

If we want to prove the analogous result for the semigroup $W_t^{(\alpha)}$ generated by the fractional Laplacian $-(-\Delta)^{\alpha}$, $0<\alpha < 1$, it is not clear how to define $W_t^{(\alpha)}u$ for $u\in S'$ since $\xi\mapsto e^{-t|\xi|^{2\alpha}}$ is not smooth at the origin, hence it is no multiplier on $S$. If we can restrict ourselves, however, to $u\in S$ or $u\in L_p$, $W_t^{(\alpha)}$ is well defined, as it is a convolution semigroup on all spaces $L_p$, $1\leq p\leq\infty$.

\begin{theorem}\label{est-05}
    Denote by $(W^{(\alpha)}_t)_{t\geq 0}$, $W_t^{(\alpha)} = e^{-t(-\Delta)^\alpha}$ the `fractional' heat semigroup of order $\alpha\in(0,1)$ generated by the fractional Laplace operator $-(-\Delta)^\alpha$. Let $s\in\real$ and $1\leq p,q < \infty$.

    With the constant $c>0$ from \eqref{est-e10} one has for all $d\geq 0$, $t>0$ and $u\in A_{p,q}^s$
    \begin{equation}\label{est-e12}
        \nnorm{W_t^{(\alpha)} u\mid A_{p,q}^{s+d}}
        \leq c \left(t^{-d/(2\alpha)} \frac{\Gamma\left(1+d/(2\alpha)\right)}{\Gamma\left(1+d/2\right)} +  1\right) \nnorm{u \mid A_{p,q}^s}.
    \end{equation}
    In particular, the fractional counterpart of \eqref{est-e10} holds for some constant $c' = c'_{p,q,s,\alpha}$
    \begin{equation}\label{est-e14}
        \nnorm{W_t^{(\alpha)} u \mid A_{p,q}^{s+d}}
        \leq c't^{-d/(2\alpha)} \nnorm{u \mid A_{p,q}^s},\quad 0< t\leq 1.
    \end{equation}
\end{theorem}
\begin{proof}
    Since $p,q<\infty$, the Schwartz functions $S$ are dense in $A_{p,q}^s$. This means that we have to prove \eqref{est-e14} only for $u\in S$.
    Using Bochner's subordination we can write
    \begin{gather*}
        W^{(\alpha)}_t u(x) = \int_0^\infty W_r u(x)\,\mu^{(\alpha)}_t(dr),\quad t>0.
    \end{gather*}
    Since the measures $\mu^{(\alpha)}_t(dr)$ are probability measures, we can use the vector-valued triangle inequality for the norm $\|u\mid A_{p,q}^s\|$ to deduce
    \begin{align*}
        \nnorm{W^{(\alpha)}_t u \mid A_{p,q}^{s+d}}
        &\leq \int_0^1 \nnorm{W_r u \mid A_{p,q}^{s+d}}\,\mu^{(\alpha)}_t(dr)
        + \int_1^\infty \nnorm{W_r u \mid A_{p,q}^{s+d}}\,\mu^{(\alpha)}_t(dr)\\
        &= \int_0^1 \nnorm{W_r u \mid A_{p,q}^{s+d}}\,\mu^{(\alpha)}_t(dr)
        + \int_1^\infty \nnorm{W_1 W_{r-1} u \mid A_{p,q}^{s+d}}\,\mu^{(\alpha)}_t(dr).
    \intertext{Using first \eqref{est-e10} for both terms (with $t=1$ in the second term), and then Theorem~\ref{contr} for the second term, yields}
        \nnorm{W^{(\alpha)}_t u \mid A_{p,q}^{s+d}}
        &\leq c\int_0^1  r^{ -d/2}\,\mu^{(\alpha)}_t(dr) \cdot \nnorm{u \mid A_{p,q}^{s}}
        + c\int_1^\infty \nnorm{W_{r-1} u \mid A_{p,q}^{s}}\,\mu^{(\alpha)}_t(dr)\\
        &\leq c\int_0^1 r^{ -d/2}\,\mu^{(\alpha)}_t(dr) \cdot \nnorm{u \mid A_{p,q}^{s}}
        + c\int_1^\infty \mu^{(\alpha)}_t(dr) \cdot \nnorm{u \mid A_{p,q}^{s}}.
    \end{align*}
    In order to estimate the integral expressions we recall that $\mu_t^{(\alpha)}(dr)$ is the transition semigroup of an $\alpha$-stable subordinator $(S^{(\alpha)}_t)_{t\geq 0}$. Therefore,
    \begin{gather*}
        \int_0^1 r^{ -d/2}\,\mu^{(\alpha)}_t(dr)
        \leq \int_0^\infty r^{ -d/2}\,\mu^{(\alpha)}_t(dr)
        = \Ee \left[\big(S_t^{(\alpha)}\big)^{ -d/2}\right]
        = t^{-d/(2\alpha)} \frac{\Gamma\left(1+d/(2\alpha)\right)}
        {\Gamma\left(1+ d/2\right)},
    \end{gather*}
    see Lemma~\ref{app-moment} in the appendix.  Since
    \begin{gather*}
        \int_1^\infty \mu^{(\alpha)}_t(dr)
        = \Pp(S_t^{(\alpha)} > 1)
        \leq 1,
    \end{gather*}
    we get \eqref{est-e12}; the estimate \eqref{est-e14} is now obvious.
\end{proof}

Observe that $A_{p,q}^s\subset L_p$ if $s>0$. If we use in the proof of Theorem~\ref{est-05} $u\in L_p$, $1\leq p\leq \infty$, instead of $u\in S$, we immediately get the following result.

\begin{corollary}\label{est-07}
    Denote by $(W^{(\alpha)}_t)_{t\geq 0}$, $W_t^{(\alpha)} = e^{-t(-\Delta)^\alpha}$ the `fractional' heat semigroup of order $\alpha\in(0,1)$ generated by the fractional Laplace operator $-(-\Delta)^\alpha$. The estimates \eqref{est-e12} and \eqref{est-e14} of Theorem~\ref{est-05} remain valid if $s> 0$ and $1\leq p,q\leq\infty$.
\end{corollary}

In order to treat the remaining cases $A_{p,q}^s$ where $s \leq 0$ and $\max\{p,q\}=\infty$ we use a lifting trick;  we are grateful to H.\ Triebel for pointing this out to us (private communication), see also the discussion in \cite[p.~104]{T20}.  Recall that the lifting operator $(1-\Delta)^{r/2}$ is a bijection between $A_{p,q}^s$ and $A_{p,q}^{s-r}$ for all $0<p,q\leq\infty$ and $s\in\real$. On the Schwartz space $S$ the lifiting operator and $W_t^{(\alpha)}$ commute,
\begin{gather*}
    W_t^{(\alpha)}u =(1-\Delta)^{-r/2} W_t^{(\alpha)} (1-\Delta)^{r/2} u\quad\text{for all $u\in S$}.
\end{gather*}
Let $s\in\real$ and pick $r$ with $s>r$. The operator $\overline W_t^{(\alpha)} := (1-\Delta)^{-r/2} W_t^{(\alpha)} (1-\Delta)^{r/2}$ is well-defined on $A_{p,q}^s$, extends $W_t^{(\alpha)}$ and makes the following diagram commutative:
\begin{gather*}
\begin{array}{ccc}
    A_{p,q}^s &\xrightarrow[\qquad\qquad\qquad\qquad\qquad]{(1-\Delta)^{r/2}} & A_{p,q}^{s-r}\\
    \llap{\text{\tiny $\overline W_t^{(\alpha)}$}}\Bigg\downarrow & &\Bigg\downarrow\rlap{\text{\tiny $W_t^{(\alpha)}$}}\\
    A_{p,q}^{s+d} &\xleftarrow[\qquad\qquad\qquad\qquad\qquad]{(1-\Delta)^{-r/2}} & A_{p,q}^{s-r+d}
\end{array}
\end{gather*}
It is not hard to see that, for any fixed  $s\in\real$, the extension $\overline W_t^{(\alpha)}$ onto $A_{p,q}^s$  does not depend on $r<s$, i.e.\ we may understand $\overline W_t^{(\alpha)}$ as an operator on $A_{p,q}^{-\infty}:=\bigcup_{s\in\real} A_{p,q}^s$. Together with the previous considerations we get

\begin{corollary}\label{est-09}
    Denote by $(\overline W^{(\alpha)}_t)_{t\geq 0}$ the `extension by lifting' of the fractional heat semigroup $W_t^{(\alpha)} = e^{-t(-\Delta)^\alpha}$ of order $\alpha\in(0,1)$. The estimates \eqref{est-e12} and \eqref{est-e14} of Theorem~\ref{est-05} remain valid for $\overline W_t^{(\alpha)}$ for all $s\in\real$ and $1\leq p,q\leq\infty$.

    If $s>0$ and $1\leq p,q\leq\infty$ or $s\in\real$ and $1\leq p,q<\infty$, these estimates are true for the original semigroup operators $W_t^{(\alpha)}$.
\end{corollary}

\section{Two extensions of the subordination technique}\label{ext}

The subordination technique which we have developed in the previous Section~\ref{est} can be extended into two directions:
(i) We may give up the concept of fractional powers in favour of general Bernstein functions, or
(ii) we may look at higher-order `fractional' semigroups $W_t^{(\beta)}$ where $\beta > 0$.

\medskip
The extension from fractional powers $\lambda\mapsto \lambda^\alpha$ to arbitrary Bernstein functions $\lambda\mapsto f(\lambda)$, see Section~\ref{sub}, is straightforward. Using general subordinate semigroups $(W_t^f)_{t\geq 0}$ instead of the fractional heat semigroup $(W_t^{(\alpha)})_{t\geq 0}$, the arguments of Section~\ref{est} go through almost literally. As before, $\overline W_t^f$ denotes the `extension by lifting' of $W_t^f$. Note that $\overline W_t^f = W_t^f$ if $\xi\mapsto f(|\xi|^2)$ is smooth at the origin. Typical examples are the `relativistic' semigroups of the form $f(\lambda) = (\lambda+1)^\alpha - 1$ for $0<\alpha < 1$.

\begin{theorem}\label{ext-07}
    Let $(W_t)_{t\geq 0}$ be as in Lemma~\ref{est-05}, let $f$ be a Bernstein function,  $(S_t^f)_{t\geq 0}$ the corresponding subordinator, and denote by  $(\overline W_t^f)_{t\geq 0}$ the subordinate semigroup extended by lifting.  For the constant $c=c_{p,q,s}$ appearing in \eqref{est-e10} and $s\in\real$, $1\leq p,q\leq\infty$ and $d\geq 0$ we have
    \begin{equation}\label{ext-e02}
        \nnorm{\overline W_t^{f} u\mid A_{p,q}^{s+d}}
        \leq c \left(\Ee \left[\big(S_t^f)^{-d/2}\right] + \Pp(S_t^f > 1)\right) \nnorm{u \mid A_{p,q}^s},\quad t > 0.
    \end{equation}
    In particular, there exists some constant $c' = c'_{p,q,s,f}$ such that for $d>0$
    \begin{equation}\label{ext-e04}
        \nnorm{\overline W_t^{f} u \mid A_{p,q}^{s+d}}
        \leq c' \Ee \left[\big(S_t^f)^{-d/2}\right] \nnorm{u \mid A_{p,q}^s},\quad 0< t\leq 1.
    \end{equation}
    If $s\geq 0$ and $1\leq p,q\leq \infty$ or $s\in\real$ and $1\leq p,q < \infty$, these estimates remain valid for the non-extended semigroup $(W_t^f)_{t\geq 0}$.
\end{theorem}
\begin{proof}
    The estimate \eqref{ext-e02} follows just as in the proof of Lemma~\ref{est-05}. In order to see \eqref{ext-e04} observe that by monotone convergence and the fact that $S_0^f=0$
    \begin{gather*}
        \lim_{t\to 0} \Ee \left[\big(S_t^f)^{-d/2}\right] = \infty
        \quad\text{and, trivially,}\quad
        \Pp(S_t^f > 1)\leq 1.
    \qedhere
    \end{gather*}
\end{proof}

Using Lemma~\ref{app-asymp} we can control the growth of the expectation appearing in \eqref{ext-e04}.

\begin{corollary}\label{ext-09}
    If, in the setting of Corollary~\ref{ext-07}, the Bernstein function $f$ satisfies $\liminf\limits_{\lambda\to 0} f(2\lambda)/f(\lambda)>1$, there is some constant $C' = C'_{p,q,s,f}$ such that
    \begin{equation}\label{ext-e06}
        [f^{-1}(1/t)]^{-d/2}\cdot \nnorm{\overline W_t^{f} u \mid A_{p,q}^{s+d}}
        \leq C' \nnorm{u \mid A_{p,q}^s},\quad 0< t\leq 1.
    \end{equation}
\end{corollary}

\begin{remark}
    Bochner's subordination is an abstract technique that works in all Banach spaces. The essential ingredient in the proof of Theorem~\ref{est-05} is the generalized triangle inequality which allows us to estimate the norm of an integral $\nnorm{\int\dots}$ by the integral of the norm $\int\nnorm{\dots}$. This shows that our results can be extended to \textbf{(tempered) homogeneous spaces} of the form ${\overset{*}{A}}\vphantom{A}^s_{p,q}(\rn)$ as well as \textbf{hybrid spaces} $A^{s,\tau}_{p,q}(\rn) = L^r A_{p,q}^s$, $\tau = p^{-1} + rn^{-1}$. The admissible parameters should be $p,q\in [1,\infty)$, $s\in\real$ and $-np^{-1}\leq r  < \infty$. As standard reference of these spaces we refer to \cite[Section 4.1, Section 1.1.2]{T20} and the literature given there.
\end{remark}

Let us now discuss higher-order generalized heat equations. In a series of papers, Baaske \& Schmei{\ss}er \cite{bas-sch17,bas-sch18,bas-sch19} studied semigroups $(W_t^{(m)})_{t\geq 0}$, $m\in\nat$, which are defined via
\begin{gather*}
    \Fcal W_t^{(m)}u(\xi) := e^{-t|\xi|^{2m}}\Fcal u(\xi),\quad u\in S,\;\xi\in\rn,\; t>0.
\end{gather*}
It is clear that $(W_t^{(m)})_{t\geq 0}$ is a semigroup which is given by a convolution kernel, $W_t^{(m)}u = K_{t,m}*u$, but while $K_{t,m}(x) = (2\pi)^{-n/2}\Fcal_{\xi\mapsto x}^{-1} e^{-|\xi|^{2m}}$ is from $S$, it may have arbitrary sign; in particular, $W_t^{(m)}$ is a uniformly bounded semigroup on $L_p$, $1\leq p<\infty$, but it is not positivity preserving. This means, in particular, that there is no Markov process which has $W_t^{(m)}$ as a transition semigroup. Nevertheless, Bochner's subordination formula \eqref{sub-e06} is still applicable; if we use $f(\lambda)=\lambda^\alpha$ for some $\alpha\in(0,1)$, we get a (in general, not positivity preserving) subordinate semigroup $(W_{t}^{(m),(\alpha)})_{t\geq 0}$. The calculation used in the proof of Lemma~\ref{sub-07} shows that
\begin{gather*}
    \Fcal W_t^{(m),(\alpha)} u(\xi) = e^{-t|\xi|^{2m\alpha}}\Fcal u(\xi) = \Fcal W_t^{(\alpha m)}u(\xi)
    \quad\text{for all $0<\alpha<1$, $m\in\nat$.}
\end{gather*}

A key result of Baaske \& Schmei{\ss}er \cite[Theorem~3.5]{bas-sch17} is the following caloric smoothing estimate for the operators $W_t^{(m)}$: {\itshape Let $1\leq p,q\leq \infty$ \textup{(}$p<\infty$ for the $F$-scale\textup{)}, $s\in\real$, $d\geq 0$ and $m\in\nat$. There is a constant $c>0$ such that
\begin{equation}\label{ext-e10}
    \|W_t^{(m)}u \mid A_{p,q}^{s+d}\|
    \leq c t^{-d/(2m)}\|u\mid A_{p,q}^s\|
    \quad\text{for all $t\in (0,1]$.}
\end{equation}}
If we use \eqref{ext-e10} instead of \eqref{est-e10} and write $\beta:=\alpha m$, we get immediately the following corollary to Theorem~\ref{est-05}.

\begin{corollary}\label{ext-11}
    Denote by $(W^{(\beta)}_t)_{t\geq 0}$, $W_t^{(\beta)} = e^{-t(-\Delta)^{\beta}}$ the generalized `fractional' heat semigroup of order $\beta>0$ generated by the higher-order fractional Laplace operator $-(-\Delta)^\beta$. Let $s\in\real$ and $1\leq p,q< \infty$.

    With the constant $c>0$ from \eqref{ext-e10} one has for every $d\geq 0$, $t>0$ and $u\in A_{p,q}^s$
    \begin{equation}\label{ext-e12}
        \nnorm{W_t^{(\beta)} u\mid A_{p,q}^{s+d}}
        \leq c \left(t^{-d/(2\beta)} \frac{\Gamma\left(1+d/(2\beta)\right)}{\Gamma\left(1+ d/2\right)} + 1\right) \nnorm{u \mid A_{p,q}^s}.
    \end{equation}
    In particular, the fractional counterpart of \eqref{est-e10} holds for some constant $c' = c'_{p,q,s,\alpha}$
    \begin{equation}\label{ext-e14}
        \nnorm{W_t^{(\beta)} u \mid A_{p,q}^{s+d}}
        \leq c't^{-d/(2\beta)} \nnorm{u \mid A_{p,q}^s},\quad 0< t\leq 1.
    \end{equation}
\end{corollary}

The cases $p=\infty, 1\leq q<\infty$ (for the $F$-scale) and $\max(p,q)=\infty$ (for the $B$-scale) are special and require the `extension by lifting' $\overline W^{(\beta)}_t$ explained at the end of Section~\ref{ext}. The analogues of \eqref{ext-e12} and \eqref{ext-e14} should be clear. If $\xi\mapsto |\xi|^{2\beta}$ is smooth, i.e.\ if $\beta\in\nat$, there is no need for an extension. At the moment, there is no subordination version for the spaces $F_{p,\infty}^s$, since in these cases \eqref{ext-e10} is yet unknown.

\section{An application of the caloric smoothing estimate}\label{apl}

The  result \eqref{ext-e10} was used in \cite{bas-sch17} to prove the existence and uniqueness of a mild solution to the non-linear equation
\begin{equation}\label{ext-e16}
\begin{split}
    \partial_t u(x,t) + (-\Delta_x )^m  u(x,t)
    &= \mathrm{div}[u^2] (x,t), \quad x\in \rn, \, t\in (0,T]\\
    u(x,0)
    &= u_0(x), \quad x\in \rn,
\end{split}
\end{equation}
where $\mathrm{div}[u^2]=\sum_{i=1}^n \frac{\partial}{\partial x_i} u^2$ is the divergence, $\Delta_x$ the Laplacian and $m\in\nat$. A \textbf{mild solution} is an element $u\in  S'(\real^{n+1})$, which is a fixed point for the operator
\begin{gather*}
    Q^{(m)} u(x,t) = W_t^{(m)} u_0(x)+ \int_0^t W_{t-\tau}^{(m)}\left(\mathrm{div}[u^2]\right)(x,\tau)\, d\tau,
    \quad x\in \rn,\, t\in (0,T)
\intertext{in the space}
    L_a \left( (0,T),b,A_{p,q}^s\right)
    := \left\{ u: (0,T)\to A_{p,q}^s, \quad \int_0^T t^{ab} \nnorm{u(\cdot,t)\mid A_{p,q}^s}^a\,  dt  < \infty\right\}
\end{gather*}
with some $a,b>0$  (and the usual modification of the norm if $a=\infty$).  A solution is called \textbf{strong}, if it is mild and if for any initial value $u_0\in A_{p,q}^{s_0}$ it belongs to $C\left( [0,T), A_{p,q}^{\alpha_0}\right)$ for some $\alpha_0$.  For suitable parameters $a,b,p,q,s$, a mild solution will be a strong solution, see \cite[Theorem 3.8.(ii)]{bas-sch17}.

The caloric estimate \eqref{ext-e10} was used in the proof of the existence of the mild solution, in order to show the contractivity of $Q^{(m)}$ and to apply a fixed point argument. Corollary~\ref{ext-11} enables us to follow the same procedure for the fractional equation
\begin{equation}\label{ext-e18}
\begin{split}
    \partial_t u(x,t) + (-\Delta_x )^\beta  u(x,t)
    &= \mathrm{div}[u^2] (x,t), \quad x\in \rn, \, t\in (0,T]\\
    u(x,0) &= u_0(x), \quad x\in \rn,
\end{split}
\end{equation}
where $\beta = m\alpha$ where $m=1,2,3,\dots$ and $\alpha\in (0,1)$; the solution is understood as an element of the space $A^{-\infty}_{p,q} = \bigcup_{s\in \real} A_{p,q}^s$. To do so, we extend the notion of a mild solution in the following way: $u(x,t)$ is a mild solution if $u(\cdot,t)\in A^{-\infty}_{p,q}$, $u(x,\cdot ) \in C^\infty(0,T)$, and $u$ is a fixed point of  $Q^{(\beta)}$. Note that $Q^{(\beta)}$ corresponds to the semigroup $W_t^{(m),(\alpha)} = W_t^{(m\alpha)}$, $t\geq 0$, obtained by subordination from $(W_t^{(m)})_{t\geq 0}$.

Corollary~\ref{ext-11} allows us to extend the result of Baaske \& Schmei{\ss}er from $m\in\nat$ to \textbf{all real} $\beta\geq 1$.
We state this result without proof; the proof of \cite[Theorem~3.8]{bas-sch17} transfers literally to the new situation.  The only change is at the very end of the proof in \cite[Eq.\ (3.79)]{bas-sch17}. Here we establish the continuity first for $u_0\in S(\rn)$ and argue then by density. Notice that $\|W_t^{(\beta)} u \mid A_{p,q}^{s-\beta+\delta}\|\leq c\|u \mid A_{p,q}^{s-\beta+\delta}\|$ by \eqref{ext-e10} with $d=0$ and $s\rightsquigarrow s-\beta+\delta$ for all $u\in S(\rn)$ with a uniform constant $c$. This is necessary since $e^{-|\xi|^\beta}$ is, in general, not a multiplier on $S(\rn)$.
The restriction $\beta\geq 1$ is needed in the proof of the contraction property \cite[proof of Theorem 3.8, Step 1]{bas-sch17}, while all other steps do work for $\beta>0$.

\begin{theorem}
    Let $n\geq 2$, $\beta\in [1,\infty)$, $1\leq p,q\leq \infty$ \textup{(}$p<\infty$ for the $F$-scale\textup{)} and $s\in\real$ is such that $A_{p,q}^s(\rn)$ is a multiplication algebra. Let
    \begin{gather*}
        a= \beta-\tfrac 1v-\beta\lambda,
        \quad\text{where}\quad
        \tfrac{2}{\beta}<v\leq \infty, \quad 0<\lambda < \epsilon\leq 1,
    \end{gather*}
    and $u_0\in A_{p,q}^{s-\beta+\beta  \epsilon}(\rn)$ be the initial data. There exists some $T>0$ such that \eqref{ext-e18} has a unique
    mild solution
    \begin{gather*}
        u\in L_{2\beta v} \big( (0,T), \tfrac{a}{2\beta}, A_{p,q}^s (\rn)\big) \cap C^\infty((0,T)\times\rn).
    \end{gather*}
    The mild solution is a strong solution if, in addition, $p,q<\infty$ and $\frac 12  \epsilon\leq \lambda< \epsilon\leq 1$ \textup{(}if $v<\infty$\textup{)}, resp., $\frac 12 \epsilon <\lambda< \epsilon \leq 1$ \textup{(}if $v=\infty$\textup{)}.
\end{theorem}

\section{Appendix -- some moment estimates}\label{app}

We need the following moment estimate for $\alpha$-stable subordinators. Although the result is well-known, see e.g.\ Sato~\cite[Eq.\ (25.5), p.\ 162]{sato13}, we include the proof for our readers' convenience. The short argument given below seems to be new.

\begin{lemma}\label{app-moment}
    Let $(S^{(\alpha)}_t)_{t\geq 0}$ be a stable subordinator with Bernstein function $f(\lambda)=\lambda^\alpha$,  $0<\alpha<1$, and transition semigroup $(\mu^{(\alpha)}_{t})_{t\geq 0}$. The moments $\Ee \left[\big(S_t^{(\alpha)}\big)^\kappa\right]$ exist for any $\kappa\in(-\infty,\alpha)$ and $t>0$. Moreover,
    \begin{gather*}
        \Ee \left[\big(S_t^{(\alpha)}\big)^\kappa\right]
        =
        \frac{\Gamma\left(1-\frac\kappa\alpha\right)}
        {\Gamma(1-\kappa)}\,t^{\frac\kappa\alpha},
        \quad t>0.
    \end{gather*}
\end{lemma}
\begin{proof}
    In this proof we write $S_t$ and $\mu_t$ instead of $S_t^{(\alpha)}$ and $\mu_t^{(\alpha)}$. Recall that the Laplace transform of $S_t$ is $\Ee\,e^{- x S_t} = \int_0^\infty e^{-xs}\,\mu_t(ds) = e^{-tx^\alpha}$, $x,t>0$. Substituting $\lambda=S_t$ in the well-known formula \cite[p.~vii]{SSV}
    \begin{gather*}
        \lambda^{-r} = \frac 1{\Gamma(r)}\int_0^\infty e^{-\lambda x} x^{r-1}\,dx,\quad \lambda > 0,\; r>0,
    \end{gather*}
    and taking expectations yields, because of Tonelli's theorem,
    \begin{gather}\label{app-e20}
        \Ee S_t^{-r}
        = \frac 1{\Gamma(r)}\int_0^\infty \Ee\,e^{-x S_t} x^{r-1}\,dx
        = \frac 1{\Gamma(r)}\int_0^\infty e^{-tx^\alpha} x^{r}\,\frac{dx}{x}.
    \end{gather}
    Now we change variables according to $y= t x^\alpha$, and get
    \begin{gather*}
        \Ee S_t^{-r}
        = \frac 1{\Gamma(r)} \frac 1\alpha t^{-\frac r\alpha}\int_0^\infty e^{-y} y^{\frac r\alpha}\,\frac{d y}{y}
        = t^{-\frac r\alpha} \frac{1}{r\Gamma(r)} \cdot \frac r\alpha\,\Gamma\left(\frac r\alpha\right)
        = t^{-\frac r\alpha} \frac{\Gamma\left(1+\frac{r}{\alpha}\right)}{\Gamma(1+r)}.
    \end{gather*}
    Setting $\kappa = -r$ proves the assertion for $\kappa\in (-\infty,0)$. This formula extends (analytically) to $-r = \kappa < \alpha$. Alternatively, we use a similar calculation and the L\'evy--Khintchine formula from Example~\ref{sub-09}
    \begin{equation*}
        \lambda^{r} = \frac r{\Gamma(1-r)}\int_0^\infty \left(1-e^{-\lambda x}\right) x^{-r-1}\,d x,\quad \lambda > 0,\; r\in (0,1),
    \end{equation*}
    to get the assertion for $\kappa\in (0,\alpha)$.
\end{proof}

The upper bound of the following lemma appears in the proof of \cite[Theorem~2.1]{dss17}.

\begin{lemma}\label{app-asymp}
    Let $(S^f_t)_{t\geq 0}$ be a subordinator with Bernstein function $f$ and transition semigroup $(\mu^{f}_{t})_{t\geq 0}$.
    Assume that $\lim_{\lambda\to\infty} f(\lambda)=\infty$ and that the inverse of $f$ satisfies $\limsup_{t\to\infty} f^{-1}(2t)/f^{-1}(t)<\infty$, (one can show, cf.\ \cite[Lemma~2.2]{dss17}, that these two conditions are equivalent to $\liminf_{\lambda\to 0} f(2\lambda)/f(\lambda)>1$), respectively.

    Under these assumptions, the moments $\Ee \left[\big(S_t^{(f)}\big)^{-r}\right]$ exist for any $r>0$ and $t>0$. Moreover,
    \begin{gather*}
        \tfrac 1{3\Gamma(1+r)}\left[f^{-1}\left(\tfrac 1t\right)\right]^{r}
        \leq
        \Ee \left[\big(S_t^{f}\big)^{-r}\right]
        \leq
         \tfrac C{\Gamma(1+r)} \left[f^{-1}\left(\tfrac 1t\right)\right]^{r},
        \quad 0<t\leq 1.
    \end{gather*}
\end{lemma}
\begin{proof}
    We write $S_t$ and $\mu_t$ instead of $S_t^f$ and $\mu_t^f$.
    Using the argument from the proof of Lemma~\ref{app-moment}, we get the following analogue of \eqref{app-e20}
    \begin{gather*}
        \Ee \left[(S_t)^{-r}\right]
        =  \frac 1{\Gamma(r)} \int_0^\infty e^{-tf(x)}\,x^r\,\frac{dx}{x}.
    \end{gather*}
    Changing variables according to $y=f(x)$---observe that $f$ is strictly monotone, $f(0)=0$ and $f(\infty)=\infty$---and using the fact that $f^{-1}(2y)\leq c f^{-1}(y)$ for, say, $y\geq 1$, we get for all $t\in (0,1]$
    \begin{gather*}
        \Ee \left[(S_t)^{-r}\right]
        =  \frac 1{r\Gamma(r)} \left(\int_0^{1/t} + \sum_{n=0}^\infty \int_{2^n/t}^{2^{n+1}/t} \right) e^{-ty}\,d_y[f^{-1}(y)]^r.
    \end{gather*}
    Since $r\Gamma(r)=\Gamma(r+1)$, this implies
    \begin{align*}
        e^{-1} [f^{-1}(1/t)]^r
        \leq \Gamma(r+1) \Ee \left[(S_t)^{-r}\right]
        &\leq  [f^{-1}(1/t)]^r + \sum_{n=0}^\infty e^{-2^n} [f^{-1}(2^{n+1}/t)]^r\\
        &\leq  \left(1+\sum_{n=0}^\infty e^{-2^n} c^{(n+1)r}\right) [f^{-1}(1/t)]^r.
    \qedhere
    \end{align*}
\end{proof}

\bigskip

\noindent
\textbf{Acknowledgement:}
    We thank Hans Triebel from Jena for his encouragement, helpful comments and the possibility to use the preprint of his new monograph~\cite{T20}.

\bigskip

\noindent
\textbf{Note added in proof:}
Meanwhile, the question just before Section 6, on the validity of (5.4) for the spaces $F^s_{p,\infty}$ has been affirmatively answered, seeF. K\"{u}hn and R.L. Schilling: \emph{Convolution inequalities for Besov and Triebel-Lizorkin spaces, and applications to convolution semigroups.} Studia Math. \textbf{262} (2022) 93--119.

\end{document}